\documentclass[11pt,oneside]{amsart}
\usepackage{geometry}                % See geometry.pdf to learn the layout options. There are lots.
\usepackage{graphicx}

\usepackage{color}  %% PUR USARE I COLORI
\usepackage{amssymb}
\usepackage{epstopdf}
\theoremstyle{plain}
\newtheorem{thm}{Theorem}[section]
\newtheorem{lem}[thm]{Lemma}
\newtheorem{prop}[thm]{Proposition}
\newtheorem{cor}[thm]{Corollary}
\newtheorem{ques}[thm]{Question}
\newtheorem{example}[thm]{Example}

\theoremstyle{definition}

\newtheorem{rem}[thm]{Remark}

\DeclareMathOperator{\N}{\mathbb N} 
\DeclareMathOperator{\Ap}{Ap}
\def\opn#1#2{\def#1{\operatorname{#2}}} % to make operators
\opn\Ap{Ap}

\title{Dilatations of numerical semigroups}

\author{V. Barucci}
\address{V. Barucci - Dipartimento di Matematica - Sapienza - Universit\`a di Roma - Piazzale A. Moro 2 - 00185 Rome - Italy}
\email{barucci@mat.uniroma1.it}
\author{F. Strazzanti}
\address{F. Strazzanti - Departamento de \'Algebra - Facultad de Matem\'aticas - Universidad de Sevilla - Avda. Reina Mercedes s/n - 41080 Sevilla - Spain}
\email{francesco.strazzanti@gmail.com}

\thanks{The second author was partially supported by MTM2016-75027-P, MTM2013-46231-P (Ministerio de Economı\'ia y Competitividad) and FEDER}


\begin{document}

\begin{abstract}
This paper is focused on numerical semigroups and presents a simple construction, that we call dilatation, which, from a starting semigroup $S$, permits to get an infinite family of semigroups which share several properties with $S$. The invariants of each semigroup $T$ of this family are given in terms of the corresponding invariants of $S$ and the Ap\'ery set and the minimal generators of $T$ are also described. 
We also study three properties that are close to the Gorenstein property of the associated semigroup ring: almost Gorenstein, 2-AGL, and nearly Gorenstein properties. 
More precisely, we prove that $S$ satisfies one of these properties if and only if each dilatation of $S$ satisfies the corresponding one.
\end{abstract}

\maketitle

 \section{Introduction}
 
A numerical semigroup $S$ is a submonoid of $\mathbb{N}$ for which $\mathbb{N} \setminus S$ is finite. Numerical semigroups arise in several context, for instance, given a field $k$, it is possible to define a ring $k[[S]]=k[[t^s ; s \in S]]$ called the numerical semigroup ring associated with S; that's why numerical semigroup theory has many analogies with commutative algebra.

The aim of the paper is to develop a new construction that we call {\em dilatation} of $S$. More precisely, given $a \in S$, we study the numerical semigroup $S+a=\{s+a ; s \in S \setminus \{0\}\} \cup \{0\}$. Actually it is enough to require that $a \in \mathbb{N} \setminus S$, provided that $S+a$ is a semigroup.
Rosales considers a similar construction in \cite{R}, taking $\{0\} \cup \{s+a; s \in S\}$, but despite the apparent similarity the two constructions are deeply different.
 
We give simple formulas for several invariants of $S+a$ (and therefore $k[[S+a]]$ as well) such as embedding dimension, type and Hilbert function, but also for the Ap\'ery set of $S+a$. We also explain how to get the minimal generators of $S+a$, giving an explicit formula if $S$ has only two generators.
Moreover, if the Wilf's conjecture holds for the starting semigroup $S$, it holds also for all the semigroups $S+a$.
 
Numerical semigroups reflect several properties of one-dimensional Cohen Macaulay rings and it is well known that Gorenstein rings correspond to symmetric semigroups. Since $S+a$ is never symmetric, we focus Section 3 on three properties which in terms of rings are close to the Gorenstein property. More precisely, we consider almost symmetric, 2-AGL, and nearly Gorenstein properties.

Almost symmetric semigroups and the corresponding almost Gorenstein rings were introduced in \cite {BF} for one-dimensional analytically unramified rings. These arise in numerical semigroup theory in order to generalize symmetric and pseudo-symmetric semigroups. 
Later the notion has been generalized to a larger class of rings, even of higher dimension (cf. \cite{GMP} and \cite{GTT}).

Nearly Gorenstein and 2-almost Gorenstein rings, briefly 2-AGL, are two very recent notions introduced in \cite{HHS} and \cite{CGKM} respectively. The 2-AGL rings are, in some sense, the rings not almost Gorenstein that are the closest to be that. Instead nearly Gorenstein rings generalize Gorenstein rings using the trace ideal and are not related to almost Gorenstein rings in general, even if  each one-dimensional almost Gorenstein ring is nearly Gorenstein.

In particular, we consider these three properties for numerical semigroups and show that a semigroup $S$ satisfies one of these properties if and only if each dilatation of $S$ satisfies the corresponding property.

Several computations of the paper are performed by using the GAP system \cite{GAP} and, in particular, the NumericalSgps package \cite{DGM}.

\section{General properties}

Let $S$ be a numerical semigroup with maximal ideal $M=S \setminus \{0\}$. We always assume that $S \neq \mathbb{N}$. For all the basic definitions of the invariants of $S$ we refer to \cite{RG}.
The elements of $\mathbb{N} \setminus S$ are called gaps of $S$ and, since there is a finite number of gaps, we denote by $g(S)$ their cardinality, which is called the genus of $S$. The maximum gap is the Frobenius number $F(S)$ of $S$.
Other important invariants are the multiplicity $e(S)$, i.e. the smallest positive element of $S$, the cardinality of $S \cap [0, F(S)]$ denoted by $n(S)$ and the type $t(S)$.  We also denote by $\Gamma (S)$ the set of minimal generators and by $\nu(S)=|\Gamma (S)|$ the embedding dimension of $S$. Moreover, let $H_S(h)$ denote the $h$-th value of the Hilbert function of $S$, i.e. $H_S(h)=|hM \setminus (h+1)M|$; recall that $H_S(1)=\nu(S)$. 

If $a \in S$, it is easy to check that $\{0\} \cup \{s+a; s \in M\}$ is also a numerical semigroup (with maximal ideal $M+a$) that we call a {\it dilatation} of $S$ with respect to $a$ and denote by $S+a$. 
The definition above can be generalized a little. In order to have a semigroup $T$ with a translation $a$ of the maximal ideal $M$ of $S$, it is necessary and sufficient that $a + s_1 + s_2 \in M$, for each $s_1,s_2 \in M$, i.e. that $a \in M - 2M$.
The semigroups $S$ and $S+a$ are strictly related:

\begin{prop}\label{basic} Let $a \in M-2M$ and $T=S+a$. Then:
\begin {enumerate}
\item $F(T)=F(S)+a$;
\item $g(T)=g(S)+a$;
\item $e(T)=e(S)+a$;
\item $n(T)=n(S)$;
\item $t(T)=t(S)+a$;
\item $H_T(h)=H_S(h)+a$ for each $h \geq 1$; 
\item $\nu(T)= \nu(S)+a$.

\end{enumerate}
\end{prop}

\begin{proof} The first four points are trivial and (7) follows from (6). For (5) observe that
the maximal ideal of $T$ is $M+a$ and $M-M=(M+a)-(M+a)$. Since $S \neq \mathbb N$,  the type of $S$ is given by $|(M-M) \setminus S|$ and  the type of $T$  is given by $|(M-M) \setminus T|$, thus applying (2) we get 
$t(T)=t(S)+a$. \\
(6) By definition  $H_S(h)=|hM \setminus (h+1)M|$ and,
since $hM+e \subseteq (h+1)M$, we have 
$$H_S(h)=|hM \setminus (hM+e)| - |(h+1)M \setminus (hM+e)|.$$
Observe that $|hM \setminus (hM+e)|=e$ and that 
$$ |(h+1)M \setminus (hM+e)|= |((h+1)M-(h+1)e) \setminus ((hM+e)-(h+1)e)|=|(h+1)(M-e) \setminus h(M-e)|$$
Denoting by $\delta_h$ this last cardinality, we have $H_S=[1,  e-\delta_1, e-\delta_2, \dots ]$.

Now, for each $h \geq 1$, 
$$h(M-e)=h((M+a)-(e+a)) $$ and so
$$\delta_h=|(h+1)(M-e) \setminus h(M-e)| = |(h+1)((M+a)-(e+a)) \setminus h((M+a)-(e+a))|.$$
The smallest $h$ such that $\delta_h=0$ is the reduction number (say $r$) of $M$ and of $M+a$.
Since $H_S=[1, e-\delta_1, e-\delta_2, \dots, e- \delta_{r-1},e,e, \dots]$ and $H_T=[1, e+a-\delta_1, e+a-\delta_2, \dots, e+a- \delta_{r-1},e+a,e+a, \dots]$, we get the thesis.
\end{proof}

It is easy to see that $S$ is of maximal embedding dimension if and only if $T$ is of maximal embedding dimension and that $S$ is an Arf semigroup if and only if $T$ is Arf. 

The Ap\'ery set of $S$ with respect to $s\in S$ is the set $\Ap(S,s)=\{x \in S ; x-s \notin S\}$. This set contains many information about the semigroup and we recall that its cardinality is always $s$.

\begin{prop} \label{Apery} Let $a \in S$. If $0< s \in S$, the Ap\'ery set of $T=S+a$ with respect to $s+a \in T$ is
\[
\Ap(T,s+a)= \{0,s+2a\} \cup \{\alpha + a \mid \alpha \in \Ap(S,s) \setminus \{0\}\} \cup \{\beta +s +a \mid \beta \in \Ap(S,a) \setminus \{0\} \}.
\]
\end{prop}

\begin{proof}
In the set in the right side there are $s+a$ elements, therefore, it is enough to prove that they are in $\Ap(T,s+a)$ and they are distinct. Clearly $0$ and $s+2a$ are in $\Ap(T,s+a)$. Let $\alpha \in \Ap(S,s)\setminus \{0\}$. Then, $\alpha+a-(s+a)=\alpha-s \notin S$ and it is not in $T$, since $T \subseteq S$. Hence, $\alpha+a \in \Ap(T,s+a)$. Now, if $\beta \in \Ap(S,a)\setminus \{0\}$, then $\beta \not \in T$ and, thus, $\beta +s +a \in \Ap(T,s+a)$.

Finally, we prove that these elements are distinct. It is clear that $0$ and $s+2a$ are not in the other two sets. Then, suppose that $\alpha + a= \beta +s +a$ with $\alpha \in \Ap(S,s)$ and $\beta \in \Ap(S,a)$, i.e. $\alpha-s=\beta$; since $\beta \in S$ and $\alpha \in \Ap(S,s)$, this yields a contradiction.
\end{proof}

In the previous proposition it is necessary that $a \in S$. In fact, consider the numerical semigroup $S=\langle 4,7 \rangle$ and $T=S+10$; it is a straightforward check that $10 \in M-2M$ and $10 \notin S$. Moreover, $\Ap(S,4)=\{7,14,21\}$, but $31=21+10$ is not in $\Ap(T,14)$, since $31-14=7+10 \in T$.

\begin{lem} \label{embedding dimension}
Let $a \in M-2M$ and $T=S+a$. If $s,s_1,s_2 \in M$, then $s_1+s_2 \in \Ap(S,s)$ if and only if $(s_1+a)+(s_2+a) \in \Ap(T,s+a)$.
\end{lem}

\begin{proof}
Assume that $s_1 + s_2 \in \Ap(S,s)$, then $(s_1+a)+(s_2+a)=(s_1+s_2+a)+a \in M+a$ and $(s_1+a)+(s_2+a)-(s+a)=(s_1+ s_2-s)+a \notin M+a$; thus $(s_1+a)+(s_2+a) \in \Ap(T,s+a)$.
Conversely, if $(s_1+a)+(s_2+a) \in \Ap(T,s+a)$, then $(s_1+a)+(s_2+a)-(s+a)=s_1+s_2-s+a \notin M+a$ and, hence, $s_1+s_2-s \notin M$, i.e. $s_1+s_2 \in \Ap(S,s)$.
\end{proof}

The previous lemma gives a one to one correspondence between $\Ap(S, e(S)) \setminus \Gamma(S)$ and $\Ap(T, e(S)+a) \setminus \Gamma(T)$. So denoting by $\epsilon$ the cardinality of such sets, we get  an alternative way to see that $\nu(T)= \nu(S)+a$, in fact $\nu (T)= (e(S)+a) - \epsilon +1= (e(S)- \epsilon+1)+a= \nu(S)+a$. Moreover, in light of Proposition \ref{Apery}, the lemma implies the following corollary:

\begin{cor} \label{Generators} 
If $a \in S$, the minimal generators of $T=S+a$ are the following:
\begin{gather*}
\{e(S)+2a\} \, \cup \, \{\alpha+a \mid \alpha \in \Ap(S,e(S)) \setminus \{0\} {\rm \ and \ } \alpha-a \notin S \setminus (\Gamma(S) \cup \{0\}) \} \, \cup \\
\cup \, \{\beta+e(S)+a \mid \beta \in \Ap(S,a) {\rm \ and \ } \beta+e(S)-a \notin S \setminus (\Gamma(S) \cup \{0\})\}.
\end{gather*}
\end{cor}

\begin{proof}
Let $x \in \Ap(T,e(S)+a)$. By the previous lemma $x$ is not a generator of $S+a$ if and only if $x-2a=s_1+s_2$ for some $s_1, s_2 \in M$, that is equivalent to $x-2a \in S \setminus (\Gamma(S) \cup \{0\})$. Recalling that all the generators except $e(S)+a$ are in $\Ap(T,e(S)+a)$, the thesis follows by Proposition $\ref{Apery}$. 
\end{proof}

In the following we give an explicit description of the minimal generators of $S+a$, when $S$ has two generators.

\begin{prop}\label{2generators}
Let $S=\langle n, m \rangle$ with $\gcd(n,m)=1$ and $n<m$. Let $a= \lambda n + \mu m \in S$ with $\mu < n$ and consider $T=S+a$. Then: \\[1mm]
$\circ$ If $\lambda=0$, the minimal generators of $T$ are \[
\{n+2a\} \cup \{ym+a \mid 1 \leq y \leq \mu +1 \} \cup \{xn+ym+a \mid 1 \leq x \leq m-1 {\rm \ and \ } 0\leq y \leq \mu-1 \};
\]
$\circ$ If $\lambda>0$, the minimal generators of $T$ are 
\begin{gather*}
%G(T)=
\{n+2a\} \cup \{ym+a \mid 1 \leq y \leq n -1\} \ \cup \\
\cup \ \{xn+ym+a \mid 1 \leq x \leq \lambda-1 {\rm \ and \ } 0\leq y \leq n+\mu-1\} \ \cup \\
\cup \ \{\lambda n+ym+a \mid 0\leq y \leq \mu+1\} \cup \{xn+ym+a \mid \lambda +1 \leq x \leq m {\rm \ and \ } 0\leq y \leq \mu-1\}.
\end{gather*}
\end{prop}

\begin{proof}
It is not difficult to prove that
\begin{equation*}
\begin{split}
\Ap(S,a)&=\{xn+ym \mid 0 \leq x \leq \lambda-1 {\rm \ and \ } 0 \leq y \leq n+\mu-1 \} \ \cup \\
 & \, \cup \{xn+ym \mid \lambda \leq x \leq m-1 {\rm \ and \ } 0 \leq y \leq \mu-1 \},
\end{split}
\end{equation*}
in particular $\Ap(S,n)=\{0,m,2m, \dots, (n-1)m\}$, as it is well known.  

Since $\nu(T)=\nu(S)+a$, in light of Corollary \ref{Generators} we only need to find $n-2$ integers in $\Ap(S,n)$ or $\Ap(S,a)$ that do not give rise to minimal generators in the sense of Corollary \ref{Generators}. 
 
If $\lambda=0$, these elements are $(\mu+2)m$, $(\mu+3)m, \dots$, $(n-1)m \in \Ap(S,n)$ and $(m-1)n$, $(m-1)n+m, \dots$, $(m-1)n+(\mu-1)m \in \Ap(S,a)$. If $\lambda>0$, they are $(\lambda-1)n+(\mu+2)m$, $(\lambda-1)n+(\mu+3)m, \dots$, $(\lambda-1)n+(n+\mu-1)m \in \Ap(S,a)$. The conclusion is a consequence of Corollary \ref{Generators}.
\end{proof}

\begin{example} \rm If $S= \langle 3,5\rangle$ and $a=10$, we get $T=S+a  =\{0, 13,15,16,18, \rightarrow\}$. As a matter of fact with the notation of Proposition \ref{2generators}, we have $n=3, m=5, \lambda =0, \mu=2$ and, according to Proposition \ref{2generators},   the minimal generators of $T$ are given by the union of the following sets:
$\{n+2a\}=\{23\}$,  $\{ym+a \mid 1 \leq y \leq 3 \}= \{15,20,25\}$ and  $\{xn+ym+a \mid 1 \leq x \leq 4 {\rm \ and \ } 0\leq y \leq  1 \}=\{13, 18, 16, 21, 19, 24, 22, 27 \}$.
\end{example}

The Wilf's conjecture (cf. \cite{W}) says that $F+1 \leq n \cdot \nu$. If we compare $S$ and $S+a$ we have:

\begin{prop} If the Wilf's conjecture holds for $S$, it holds for $S+a$ for all $a \in M-2M$.

\end{prop}

 \begin{proof} Let $T=S+a$ and suppose that $F(S)+1 \leq n(S) \cdot \nu(S)$. Applying Proposition \ref{basic}, we get $F(T)+1= F(S)+a+1 \leq n(S) \cdot \nu(S) +a \leq n(S) \cdot \nu(S) +n(S)a=n(S) \cdot (\nu(S)+a)= n(T) \cdot \nu(T)$.
\end{proof}

In \cite{FH} and \cite{S} it is proved that the Wilf's conjecture holds provided that $g(S) \leq 60$ or $e(S) \leq 8$ respectively. Clearly if $S$ satisfies one of this properties and $a$ is large enough, $S+a$ does not satisfy it. 

\begin{cor} If $a \in M-2M$ and either $g(S) \leq 60$ or $e(S) \leq 8$, then the Wilf's conjecture holds for $S+a$. 
\end{cor}

It would be interesting to relate other properties of $k[[S+a]]$ to the ones of $k[[S]]$. If $S= \langle n_1, \dots, n_{\nu} \rangle$, there is a surjective a map $k[[x_1, \dots, x_\nu]] \mapsto k[[S]]$ defined by $x_i \mapsto t^{n_i}$ and, therefore, $k[[S]] \cong k[[x_1, \dots, x_n]]/I_S$ for some ideal $I_S$. Clearly the ideal $I_{S+a}$ reflects the properties of $k[[S+a]]$ and, thus, it is important to understand its properties. As a first step we raise the following question based on some computations:

\begin{ques} Let $\mu(I_S)$ be the number of the minimal generators of $I_S$. Is it true that
$
\mu(I_{S+a})=\mu(I_S)+a\nu(S)+\binom{a}{2}?
$
\end{ques}

It is easy to see that the previous question has an affermative answer if $S$ has maximal embedding dimension $e$; in fact in this case it is known that $\mu(S)=e(e-1)/2$ and $\mu(S+a)=(e+a)(e+a-1)/2$, see \cite[Theorem 8.30]{RG}.

%\begin{rem} 
%If $(R,m)$ is a one-dimensional analytically irreducible ring with integral closure $(V, tV)$ and with $v(R)=S$, and if $a \in S$, a similar "dilatation" could be defined for rings taking a pullback of $R/t^am$.
%\end{rem} 

%%%%%%%%%%%%%%%%%%%%%%%%%%%%%%%%%%%%%%%%%%%%%%%%%%%%%%%
\section{Generalizations of the symmetric property}

Given a numerical semigroup $S$, the canonical ideal of $S$ is defined as the set $\Omega_S=\{x \in \mathbb{N} ; F(S)-x \notin S\}$ and $S$ is said to be symmetric if $S=\Omega_S$. This is a central notion in numerical semigroup theory and corresponds to the Gorensteinness of the numerical semigroup ring associated with $S$. It is well-known that $S$ is symmetric if and only if has type one and thus, if $a$ is positive, $S+a$ is never symmetric by Proposition \ref{basic}(5).
On the other hand, it is possible to use this construction to find numerical semigroups that are, in some sense, {\it near} to be symmetric. In particular, in this section we consider the following properties: almost symmetric, $2$-almost Gorenstein and  nearly Gorenstein. 
 
By definition a numerical semigroup $S$ with maximal ideal $M$ and canonical ideal $\Omega$ is almost symmetric if   $\Omega + M \subseteq M$ or, equivalently, if $ \Omega \subseteq M-M$. 
Since $0 \in \Omega$, we have that $\Omega \subseteq 2\Omega \subseteq 3\Omega \subseteq \dots$. The reduction number of $\Omega$ is the smallest $h$ such that $h\Omega=(h+1)\Omega$, i.e. it is the smallest $h$ such that $h\Omega$ is a semigroup.
A semigroup $S$ is almost symmetric but not symmetric if and only if the reduction number of $\Omega$ is 2 and $|2\Omega \setminus \Omega|=1$ (cf. \cite[Proposition 14]{BF}).

Now we show that the almost symmetric property is stable with respect to dilatations.

\begin{lem}Let $a \in M-2M$ and $T=S+a$. Then:
\begin{enumerate}
\item $\Omega_T=(\Omega_S \cup \{F(S)\}) \setminus \{F(T)\}$;
\item $\Omega_S=(\Omega_T \cup \{F(T)\} )\setminus \{F(S)\}$.
\end{enumerate}
\end{lem}

\begin{proof}   Suppose that $x \in \mathbb Z$,  $x \neq F(S), F(T)$. We have that $F(S)-x=F(T)-a-x \notin S$ if and only if $(F(T)-a-x)+a=F(T)-x \notin T$; then  $x\in \Omega_S$ if and only if $x\in \Omega_T$. 
Moreover, since $F(T)-F(S)=a \notin T$, we get that $F(S) \in \Omega_T$ and, obviously, $F(T) \in S \subseteq \Omega_S$; hence, the thesis follows. 
\end{proof}

\begin{prop}\label {almostsim} Let $a \in M-2M$ and $T=S+a$. Then, $S$ is almost symmetric if and only if $T$ is almost symmetric.
\end{prop}

\begin{proof} $S$ is almost symmetric if and only if $\Omega_S \subset M-M$ and $T$ is almost symmetric if and only if $\Omega_T \subset (M+a)-(M+a)=M-M$. Since $F(S)$ and $F(T)$ are always in $M-M$, we end by the previous lemma.
\end{proof}

\begin{example} \rm $S=\langle 11,14,18,20,21,23,24,27,30 \rangle$ is almost symmetric with type 8 and embedding dimension 9. If $a=5$, we get $$T= S+a=\langle 16,19,23,25,26,27,28,29,30,33,34,36,37,40\rangle$$ which is also almost symmetric (of type $8+5=13$ and embedding dimension 14).
\end{example}

By Proposition \ref{almostsim} we get that, if $S$ is symmetric, every semigroup obtained by a dilatation of $S$ is almost symmetric. However, there are almost symmetric semigroups which are not dilatations of symmetric semigroups. For instance, the semigroups of the example above are not dilatations of symmetric ones or consider, more simply,  $T= \langle 4,7,9 \rangle=\{0,4,7,8,9,11 \rightarrow\}$, which is almost symmetric and is not the dilatation of any semigroup.

Indeed it is natural to define a ``contraction'' to pass from a semigroup $T$ of maximal ideal $M_T$ to a semigroup $S=T-a$, where $a \in \N$. In this case, in order to have $S$ additively closed,  it is necessary and sufficient that $t_1 + t_2 -a \in M_T$, for each $t_1,t_2 \in M_T$, i.e. that $2M_T \subseteq M_T+a$, in particular $a \leq e(T)$.

\bigskip
Let $R$ be a one-dimensional Cohen-Macaulay local ring and let $I$ be a canonical ideal of $R$. Let $e_i(I)$ denote the Hilbert coefficients of $R$ with respect to $I$. It is known that $s=e_1(I)-e_0(I)+\ell_R(R/I)$ is positive and independent of the choice of $I$; moreover, $R$ is almost Gorenstein if and only if $s=1$. 
In \cite{CGKM} Chau, Goto, Kumashiro, and Matsuoka study the rings for which $s=2$ that they call 2-almost Gorenstein local rings or briefly 2-AGL rings. If $\omega$ is a canonical module of $R$ such that $R \subseteq \omega \subseteq \overline{R}$, where $\overline R$ denotes the integral closure of $R$, they prove that $R$ is 2-AGL if and only if $\omega^2=\omega^3$ and $\ell_R (\omega^2/\omega)=2$, see \cite[Theorem 1.4]{CGKM}. 

Similarly, given a numerical semigroup $S$ with canonical ideal $\Omega$, we say that $S$ is 2-AGL if the reduction number of $\Omega$ is 2   and $|2\Omega  \setminus \Omega |=2$. Clearly, $S$ is 2-AGL if and only if $k[[S]]$ is 2-AGL.

From what we recalled above, it is clear that almost symmetric semigroups and 2-AGL semigroups are disjoint classes.

We are interested to the property 2-AGL of $S+a$ and, therefore, we can exclude that $S$ is symmetric, since in this case $S+a$ is almost symmetric and then not 2-AGL.

\begin{lem} Let $S$ be not symmetric and let $a \in M-2M$. Then, for every integer $i \geq 2$, it holds that    $i \Omega_S=i \Omega_{S+a}$. In particular, the canonical ideals of $S$ and $S+a$ have the same reduction number.
\end{lem}

\begin{proof}
Let $x=x_1 + \dots + x_i \in i\Omega_{S+a}$ with $x_1, \dots, x_i \in \Omega_{S+a}=(\Omega_S \cup \{F(S)\}) \setminus \{F(S+a)\}$. Assume that $x_j\notin \Omega_S$ for some $j$, i.e. $x_j=F(S)$; then, $x$ is equal to or grater than $F(S)$. In the latter case $x \in S \subseteq \Omega_S \subseteq i\Omega_S$, while if $x=F(S)$ it follows that $x \in 2\Omega_S \subseteq i\Omega_S$, since $S$ is not symmetric. The other inclusion is analogous.
\end{proof}

Since $S$ and $S+a$ are not symmetric, $F(S) \in 2\Omega_S$ and $F(S+a) \in 2\Omega_{S+a}$; moreover, $F(S+a) \in \Omega_S$ and $F(S) \in \Omega_{S+a}$. Therefore, the previous lemma immediately implies the following:

\begin{cor}
Let $a \in M-2M$. Then, $S+a$ is 2-AGL if and only if $S$ is 2-AGL.
\end{cor}

Another generalization of Gorenstein ring is the notion of nearly Gorenstein ring, introduced by Herzog, Hibi, and Stamate in \cite{HHS}. In the one-dimensional case nearly Gorenstein rings generalize almost Gorenstein rings, see \cite[Proposition 6.1]{HHS}. In particular, in \cite{HHS} it is also introduced the notion of nearly Gorenstein numerical semigroups that is a generalization of almost symmetric semigroups.

The trace ideal of $S$ is defined as ${\rm tr}(S)=\Omega_S + (S-\Omega_S)$. Then $S$ is said to be nearly Gorenstein if $M \subseteq {\rm tr}(S)$. The semigroup $S$ is symmetric if and only if ${\rm tr}(S)=S$, otherwise $S$ is nearly Gorenstein exactly when ${\rm tr}(S)=M$, since $tr(S)$ is an ideal contained in $S$.
%Since   ${\rm tr}(S)$ is an ideal contained in $S$, a nearly Gorenstein semigroup  $S$ is symmetric if and only if ${\rm tr}(S)=S$, otherwise ${\rm tr}(S)=M(S)$.

We give here a simple direct proof that nearly Gorenstein semigroups generalize the almost symmetric ones.

\begin{prop} \label {easy}Each almost symmetric semigroup is nearly Gorenstein.
\end{prop}
\begin{proof}  If $S$ is symmetric, ${\rm tr}(S)=S$ and, then, it is nearly Gorenstein. If $S$ is a non-symmetric almost symmetric semigroup, we have $S-\Omega_S = M$, since $\Omega_S \subseteq M-M$. It follows that ${\rm tr}(S)=\Omega_S+ (S-\Omega_S)= \Omega_S+M=M$. 
\end{proof}

\begin{rem} \rm
Nearly Gorenstein and 2-AGL numerical semigroups are two disjoint classes. This is already noted in \cite{CGKM} for one-dimensional rings, but we include a simple proof in the case of numerical semigroups.

Assume by contradiction that $S$ is both 2-AGL and nearly Gorenstein and let $x\in M$ such that $F(S)-x \in 2\Omega\setminus \Omega$. Clearly $S$ is not symmetric and  $2\Omega \setminus \Omega=\{F(S)-x, F(S)\}$. Since $M=\Omega+(S-\Omega)$, we have $x=(F(S)-f)+s$ for some $f \notin S$ and some $s \in S-\Omega \subseteq M$; therefore, $f-s=F(S)-x \in 2\Omega\setminus \Omega$.
Consequently $F(S)-s=(F(S)-f)+f-s \in 3\Omega=2\Omega$ and it is not in $\Omega$, since $s\in M$; thus, $s=x$.
Moreover, if $F(S)-s=k_1+k_2$ with $k_1,k_2 \in \Omega$, it follows that $F(S)-(s+k_1)=k_2 \in \Omega$, but $s+k_1 \in (S-\Omega)+\Omega \subseteq S$ yields a contradiction.
\end{rem}

\begin{lem} Let $a \in M-2M$. If $S$ is not symmetric, then
${\rm tr}(S+a)={\rm tr}(S)+a$.
\end{lem}

\begin{proof} 
Let $T=S+a$. We first notice that $S-\Omega_S=S-(\Omega_S \setminus \{F(T)\})$, since $F(T)$ is in the conductor of $S$. Similarly, we claim that $T-\Omega_T=T-(\Omega_T \setminus \{F(S)\})$. The type of $T$ is greater than two and, then, there are at least two elements in $\Omega_T \setminus T$; this implies that $0 \notin T-(\Omega_T \setminus \{F(S)\})$. Let $x \in T-(\Omega_T \setminus \{F(S)\})$. Since $0 \in \Omega_T$, it follows that $x \in T \setminus \{0\}$ and, then, $x=m+a$ with $m \in M$. Therefore, $x+F(S) > F(T)$ and $x \in T-\Omega_T$. 

Moreover, ${\rm tr}(T)=(\Omega_T \setminus \{F(S)\})+(T-\Omega_T)$; in fact, if $x=F(S)+y$ with $y \in T-\Omega_T$, then we can write $x=0+(F(S)+y) \in \Omega_T \setminus \{F(S)\}+(T-\Omega_T)$, since $y \neq 0$. A similar argument shows that ${\rm tr}(S)=(\Omega_S \setminus \{F(T)\}) + (S-\Omega_S)$.

Since $S$ is not symmetric, $0 \notin (S-(\Omega_S \setminus \{F(T)\}))$ and, using this fact, it is easy to prove that $T-(\Omega_S \setminus \{F(T)\})=a + (S-(\Omega_S \setminus \{F(T)\}))$.

Hence, it follows that
\begin{gather*}
{\rm tr}(T)= (\Omega_T \setminus \{F(S)\}) + (T-\Omega_T) = (\Omega_T \setminus \{F(S)\}) + (T-(\Omega_T \setminus \{F(S)\}))= \\
=(\Omega_S \setminus \{F(T)\}) + a + (S-(\Omega_T \setminus \{F(S)\}))= a+ (\Omega_S \setminus \{F(T)\}) + (S- \Omega_S)=a+{\rm tr}(S)
\end{gather*}
\end{proof}

Recall that if $S$ is symmetric,     $S+a$ is always almost symmetric (Proposition \ref{almostsim}) and then by Proposition \ref{easy}  it is nearly Gorenstein. Thus, we get the following:

\begin{cor}
$S$ is nearly Gorenstein if and only if $S+a$ is nearly Gorenstein for all $a \in M-2M$.
\end{cor}

\end{document}